\documentclass[12pt,oneside]{amsart}

\usepackage{fullpage}
\usepackage{amssymb}
\usepackage{latexsym}
\usepackage{amsmath}
\usepackage{mathrsfs}
\usepackage[all]{xy}
\usepackage{enumerate}
\usepackage{amsthm}
\usepackage{psfrag}
\usepackage{ifthen}
\usepackage{comment}
\usepackage{caption}
\newtheorem{theorem}{Theorem}[section]
\newtheorem{lemma}[theorem]{Lemma}
\newtheorem{metalemma}[theorem]{Meta Lemma}
\newtheorem{proposition}[theorem]{Proposition}
\newtheorem{corollary}[theorem]{Corollary}

\theoremstyle{definition}
\newtheorem{definition}[theorem]{Definition}

\newtheorem{example}[theorem]{Example}

\newtheoremstyle{case}{}{}{}{}{}{:}{ }{}
\theoremstyle{case}
\newtheorem{case}{\bf{Case}}

\newcommand{\bel}[1]{\begin{equation}\label{#1}}
\newcommand{\ee}{\end{equation}}

\newcommand{\LBA}{\left\{\begin{array}}
\newcommand{\EAR}{\end{array}\right.}

\def\ovu{{\overline{u}}}

\def\ovx{{\overline{x}}}
\def\ovy{{\overline{y}}}
\def\ovz{{\overline{z}}}

\def\ovdelta{{\overline{\delta}}}

\def\CC{{\mathcal C}}
\def\CE{{\mathcal E}}
\def\CL{{\mathcal L}}
\def\CN{{\mathcal N}}
\def\CS{{\mathcal S}}
\def\CI{{\mathcal I}}
\def\CJ{{\mathcal J}}
\def\CU{{\mathcal U}}
\def\CM{{\mathcal M}}

\def\CT{{\mathcal T}}

\def\NP{{\mathbf{NP}}}
\def\DP{{\mathbf{DP}}}
\def\CQ{{\overline{\mathbb{Q}}}}

\makeatletter
\def\blfootnote{\xdef\@thefnmark{}\@footnotetext}
\makeatother

\newcommand{\gp}[1]{{\left\langle #1 \right\rangle}}

\newcommand{\rb}[1]{{\left( #1 \right)}}
\newcommand{\abs}[1]{\left|{}#1\right|}

\newcommand{\Set}[2]{\left\{\, #1 \;\middle|\; #2 \,\right\}}

\def\MN{{\mathbb{N}}}
\def\MZ{{\mathbb{Z}}}
\def\MQ{{\mathbb{Q}}}
\def\MR{{\mathbb{R}}}
\def\BS{{\mathbf{BS}}}

\DeclareMathOperator{\lcm}{{lcm}}
\DeclareMathOperator{\size}{{size}}
\DeclareMathOperator{\spn}{{span}}
\DeclareMathOperator{\mspn}{{max\;\!span}}

\DeclareMathOperator{\gap}{gap}
\DeclareMathOperator{\height}{height}
\DeclareMathOperator{\TPART}{3PART}

\title{The Diophantine problem for systems of algebraic equations with exponents}
\author{Richard Mandel}
\address{Department of Mathematics, University of the Basque Country, Bilbao, Spain}
\email{mandel.richard@ehu.eus}
\thanks{The first author was partially supported by the Basque Government Grant IT1483-22 and the Spanish Government grants PID2019-107444GA-I00 and PID2020-117281GB-I00.}

\author{Alexander Ushakov}
\address{Department of Mathematical Sciences, Stevens Institute of Technology, Hoboken NJ 07030}\email{aushakov@stevens.edu}

\date{}

\begin{document}

\maketitle

\begin{abstract}
Consider the equation $q_1\alpha^{x_1}+\dots+q_k\alpha^{x_k} = q$, with constants $\alpha \in \overline{\mathbb{Q}} \setminus \{0,1\}$, $q_1,\ldots,q_k,q\in\overline{\mathbb{Q}}$ and unknowns $x_1,\ldots,x_k$, referred to in this paper as an \emph{algebraic equation with exponents}. We prove that the problem to decide if a given equation has an integer solution is $\textbf{NP}$-complete, and that the same holds for systems of equations (whether $\alpha$ is fixed or given as part of the input). Furthermore, we describe the set of all solutions for a given system of algebraic equations with exponents and prove that it is semilinear.
\end{abstract}


\section{Introduction}

A classical \emph{Diophantine equation} is an equation with integer coefficients and one or more unknowns, for which only the integer solutions are of interest. Such equations have been studied since antiquity, motivating much fruitful work in diverse areas of mathematics (Fermat's last theorem is a particularly famous example). For a class $\CC$ of Diophantine equations (or of systems of Diophantine equations), one can study three related algorithmic problems:
\begin{enumerate}[(i)]
\item
decide whether a given equation from $\CC$ has a solution or not;
\item
find a solution for a given equation from $\CC$, assuming one exists;
\item
describe the set of all solutions for a given equation from $\CC$.
\end{enumerate}
The first problem is called
the \emph{Diophantine problem for} $\CC$, and is denoted $\DP_\CC$. In modern mathematics, the notion is extended to any algebraic structure (or a class of algebraic structures) $\CS$, where the notation $\DP_{\CC}(\CS)$ may be used. For instance, we may speak of the Diophantine problem  $\DP_{\mathcal P}(\MQ)$
for systems of polynomial equations over $\mathbb Q$, or
the Diophantine problem $\DP(\CN_3)$
for systems of arbitrary equations in the class $\CN_3$ of nilpotent groups of step $3$.

The most well studied classes of Diophantine equations are
the classes $\CL$ of linear systems and $\mathcal P$ of polynomial systems.
The problem $\DP_{\CL}$ has an efficient (polynomial time) solution:
it may be solved by computing the Hermite (or Smith) normal form of
the corresponding matrix (see \cite{Storjohann-Labahn:1996}).
On the other hand, $\DP_{\mathcal P}(\MZ)$ is Hilbert's famous tenth problem, first proposed in 1900 and only shown to have no algorithmic solution in 1970, as a result of the combined work of Y. Matiyasevich, J. Robinson, M. Davis, and H. Putnam, spanning 21 years (with Matiyasevich completing the theorem---known as the MRDP theorem---in 1970 \cite{Matiyasevich_book}). Notably, the decidability of $\DP_{\mathcal P}(\MQ)$ remains
an open question.

In this paper we study the class $\CE$ of systems of Diophantine equations of the following form:
\begin{equation}\label{eq-exp-sys}
\left\{
\begin{array}{cl}
q_{11}\alpha_1^{x_1}+\dots+q_{1k}\alpha_1^{x_k}&= q_{10}\\
\vdots&\\
q_{s1}\alpha_s^{x_1}+\dots+q_{sk}\alpha_s^{x_k}&= q_{s0}\\
\end{array}
\right.
\end{equation}
with constants $\alpha_1,\ldots,\alpha_s\in\CQ\setminus\{0,1\}$ (where $\CQ$ denotes an algebraic closure of $\MQ$), $q_{ij}\in \MQ(\alpha_i)$, unknowns $x_1,\ldots,x_k$, and whose solutions must be in $\MZ^k$.
We call these equations \emph{algebraic equations with exponents}. The Diophantine problem for $\CE$ may be considered in two distinct forms: the \emph{uniform problem}, where the $\alpha_i$ are given as part of the input, and the \emph{fixed base} problem, where the $\alpha_i$ are fixed beforehand.

Notice that there is no loss of generality in requiring that the coefficients for the $i$th equation be contained in $\MQ(\alpha_i)$. In fact, the more general problem in which $q_{i1}, \ldots, q_{ik}\in \CQ$ (with the input given in a natural form described below) is polynomial-time equivalent to $\DP_{\CE}$. This may be seen from the example of a single equation
\begin{equation}\label{eq:single-eq}
    q_1\alpha^{x_1}+\cdots+q_k\alpha^{x_k} = q_0
\end{equation}
in the following manner. Let $K/\MQ(\alpha)$ be a proper finite degree extension such that $q_0,q_1, \ldots,q_k\in K$, and let $m = [K:\MQ(\alpha)]$. Let $\beta$ be a primitive element which generates $K$ over $\MQ(\alpha)$, and suppose that each $q_i$ is given as a vector $(f_{i1},\ldots,f_{im})\in \MQ(\alpha)^m$ such that $q_i = \sum_{j=1}^{m}f_{ij}\beta^{j-1}$. Then \eqref{eq:single-eq} has a solution if and only if the system
\begin{equation}\label{eq:system-eq}
\left\{
\begin{array}{cl}
f_{11}\alpha^{x_1}+\dots+f_{k1}\alpha^{x_k}&= f_{01}\\
\vdots&\\
f_{1m}\alpha^{x_1}+\dots+f_{km}\alpha^{x_k}&= f_{0m}\\
\end{array}
\right.
\end{equation}
has a solution. Observe that \eqref{eq:system-eq} is an instance of $\DP_{\CE}$ whose input is roughly the same size as that of \eqref{eq:single-eq}. Moreover, it is clear that we can transform one into the other in polynomial time.

\subsection{Encoding the input of $\DP_{\CE}$}\label{se:encoding}
The constants $\alpha_1,\ldots,\alpha_s$ are algebraic numbers (other than 0 and 1) called the \emph{bases} of the system; in what follows, we use the notation $\deg(w)$ to denote the degree of the minimal polynomial of an algebraic number $w\in \CQ$. An instance of $\DP_{\CE}$ is given with the following data.
\begin{itemize}
\item The bases $\alpha_i$ are specified (either as part of the input, or fixed in advance) as lists of the integer coefficients of their minimal polynomials $p_{\alpha_i}(x) = \sum_{j=0}^{d_i} c_{ij}x^j$, where $d_i = \deg(\alpha_i)$.

\item The coefficients $q_{ij}\in \MQ(\alpha_i)$ are given as vectors $(r_{ij}^0,\ldots,r_{ij}^{d_i-1})\in \MQ^{d_i}$ such that $q_{ij} = \sum_{h=0}^{d_i-1}r_{ij}^h\alpha^{h}$, with $r_{ij}^h$ given as a quotient (in lowest terms) $r_{ij}^h=\frac{a_{ij}^h}{b_{ij}^h}$.
\end{itemize}
This naturally defines the size of an instance $E$ of the problem as
\begin{equation}\label{eq:size}
\size(E) = \sum_{i=1}^s\sum_{j=1}^k\sum_{h=0}^{d_i-1}
\Big(
\log_2(|a_{ij}^h|+1) + \log_2(|b_{ij}^h|+1)
\Big)+ \sum_{i=1}^s\sum_{j=0}^{d_i}\log_2(|c_{ij}|+2),
\end{equation}
if $E$ is an instance of the uniform problem, and
\begin{equation}
\size(E) = \sum_{i=1}^s\sum_{j=1}^k\sum_{h=0}^{d_i-1}
\Big(
\log_2(|a_{ij}^h|+1) + \log_2(|b_{ij}^h|+1)
\Big)
\end{equation}
if $E$ is an instance of the fixed base problem. This is roughly the number of bits required to encode $E$ (the terms $\log_2(|c_{ij}|+2)$ are needed to count the zero coefficients of $p_{\alpha_i}$, and to ensure that $\deg (\alpha_i)<\size(E)$).

The minimal polynomials of $\alpha_i$ are assumed to be irreducible, which means that $\DP_{\CE}$, as defined above, is a \emph{promise problem} (i.e. the input is ``promised'' to belong to a certain subset of all possible inputs). However, this is only a matter of convenience, since D. G. Cantor has shown that the problem of deciding whether a given polynomial is irreducible is in $\NP$ (see \cite{Cantor:1981}). Thus, our proof that $\DP_{\CE} \in \NP$ can easily be adapted to allow for arbitrary input by combining Cantor's $\NP$-certificate with our own.

We assume a basic background in complexity theory and $\NP$-completeness, as may be found in \cite{Garey-Johnson:1979}.

\subsection{Main results}
We prove that the Diophantine problem for
$\CE$ is $\NP$-complete, and that it remains $\NP$-complete for single equations
with any fixed choice of $\alpha\in\CQ\setminus\{0,1\}$. Additionally, we describe the solution set for a given instance,
and prove that it is semilinear. On a finer scale, we show that the Diophantine problem for a system $E\in\CE$ can be solved in polynomial time as long as the number of equations is bounded by a fixed constant. In other words, it is a \emph{fixed-parameter tractable problem}.

\subsection{Related results and applications}
Equations of a similar form were considered by A. Semenov in 1984, in the course of his investigation of the first-order theory of the algebraic structure $\gp{\MN;+,k^x}$ (where $k \in \MN$ is a fixed base for the exponential). Using methods involving quantifier elimination, Semenov proved the decidability of this theory in \cite{Semenov:1984}. Notably, while it is a simple matter to solve \eqref{eq-exp-sys} over the real or complex numbers, the first-order theory of $\gp{\mathbb C;+,k^x}$ is undecidable (since $\gp{
\MZ;+,\cdot}$ is interpretable in this structure), and the decidability of $\gp{\MR;+,k^x}$ is an open question (but known to follow from Schanuel's conjecture; see \cite{Marker:1996}).

Equations with exponents play a very important role in study
of discrete optimization problems
(such as the knapsack problem, the power word problem,
see \cite{Dudkin-Treyer:2018, Lohrey:2020})
and decision problems over certain groups. In fact, our present interest in the class of equations $\CE$ arose from studying the Diophantine problem over the Baumslag-Solitar groups. There turns out to be a close connection
between algebraic equations with exponents in base
$\frac{m}{n}$ and equations in $\BS(m,n)$, suggesting the
importance of decidability and complexity results for $\DP_{\CE}$. In a forthcoming paper \cite{Mandel-Ushakov:2022}, we employ some of the present results to prove, for instance, that the quadratic Diophantine problem for $\BS(1,n)$ is $\NP$-complete for $n \neq 1$, and we expect further applications to the study of equations (quadratic and otherwise) in $\BS(m,n)$ and related groups. For background and recent results on the Diophantine problem in various classes of groups, see, for instance, \cite{Kharlampovich-Taam:2017} and \cite{G-M-O-solvable:2020}, as well as \cite{Ciobanu-Holt-Rees:2020} for the decidability of the Diophantine problem in groups that are virtually a direct product of hyperbolic groups (including the unimodular Baumslag-Solitar groups $\BS(n,\pm n)$).

\subsection{Outline and general approach}
In Section~\ref{se:semenov-np-hardness}, we prove that the Diophantine problem
for a single equation with $\alpha\in\CQ\setminus\{0,1\}$
is $\NP$-hard. This is accomplished via a reduction of either the \emph{partition problem} (in case $\alpha$ is a root of unity) or the \emph{3-partition problem} (both well-known $\NP$-complete problems) to $\DP$.
In Section~\ref{se:semenov-np-equation},
we prove that the Diophantine problem
for a single equation with $\alpha\in\CQ\setminus\{0,1\}$
belongs to $\NP$.

Sections \ref{se:semenov-np-homo-system} and \ref{se:semenov-np-non-homo-system} extend the above results to finite systems of equations, and in Section~\ref{se:semenov-solution-set} we describe the solution set for a system of equations. Finally, Section~\ref{se:semenov-parametrized-complexity}
addresses the parameterized complexity of the problem.

\section{Complexity lower bound for a single equation}
\label{se:semenov-np-hardness}

Fix $\alpha\in\CQ\setminus\{0,1\}$ and consider
the equation
\begin{equation}\label{eq:semenov}
q_1\alpha^{x_1}+\dots+q_k\alpha^{x_k} = q_0
\end{equation}
with $q_1,\ldots,q_k,q_0\in \mathbb Q(\alpha)$, unknowns $x_1,\ldots,x_k$
and solutions sought in $\MZ$. In this section, we prove that the Diophantine problem for \eqref{eq:semenov} is $\NP$-hard. We first handle the special case where $\alpha$ is a root of unity.

\subsection{$\alpha$ is a root of unity}
Let $\alpha=e^{2\pi i/n} \neq 1$ for $n\ge 2$. Below, we prove $\NP$-hardness by reducing
the partition problem (which is known to be $\NP$-complete) to the decidability
of \eqref{eq:semenov}.
Recall that the \emph{partition problem} is
the problem to decide if a given multiset $\{q_1,\ldots,q_k\}$
of positive integers
can be partitioned into two submultisets $S_0$ and $S_1$
such that $\sum_{x\in S_0}x = \sum_{x\in S_1}x$, see \cite{Garey-Johnson:1979}.

For an instance $Q=\{q_1,\ldots,q_k\}$ of the partition problem, define $L=\frac{1}{2}\sum q_i$
(we may assume that $L\in\MZ$), and let $E_Q$ denote the following equation:
\begin{equation}\label{eq:root-unity}
\sum q_i \alpha^{x_i} = L+L\alpha.
\end{equation}
Note that $\size(E_Q) = \sum_{i=1}^k \log_2(q_i+1) + 2\log_2(L+1)$, which is linear in the size of $Q$ (with $Q$ represented in binary).

\begin{proposition}
$Q$ is a positive instance of the partition problem if and only if $E_Q$ has a solution.
\end{proposition}

\begin{proof}
Consider two cases.
Suppose that $\alpha=-1$. Then the right-hand side of \eqref{eq:root-unity}
is equal to zero, $\alpha^{x_i}=\pm 1$,
and \eqref{eq:root-unity} translates to the equation
$$
\sum \varepsilon_i q_i = 0,
$$
with unknowns $\varepsilon_1\dots,\varepsilon_k=\pm1$, which is equivalent to the
partition problem. The converse is also true.

Suppose that $\alpha\ne -1$. If $Q$ is a positive instance, then it is clear that $E_Q$
has a solution (with $x_i\in \{0,1\}$). Conversely, assume that $E_Q$ has a solution $\ovx = (x_1,\ldots,x_k)$.
Consider the complex number $u = 1+\alpha$ as a vector in $\MR^2$,
and let $\pi_u(z)$ denote the signed scalar projection of $z$ onto $u$,
defined by
$$
\pi_u(z)=\tfrac{\Re(\ovu z)}{|u|}.
$$
It is easily verified that
$$
\pi_u(1) = \pi_u(\alpha) = \tfrac{|u|}{2} >
\max\{\pi_u(\alpha^2),\dots,\pi_u(\alpha^{n-1})\}.
$$
Since the projection of the right-hand side of \eqref{eq:root-unity}
onto $u$ is equal to $L|u|$, and since $\sum_{i=1}^k q_i=2L$, it follows that the $x_i$ must all be either 0 or 1 (otherwise the left-hand side would have a projection strictly less than $L|u|$). Letting $S_0 = \{i\in \{1,\ldots,k\}|x_i=0\}$ and $S_1 = \{i\in \{1,\ldots,k\}|x_i=1\}$, the $\MQ$-linear independence of $\{1,\alpha\}$ implies that $$\sum_{i \in S_0}q_i = L = \sum_{i\in S_1}q_i,$$ proving that $Q$ is a positive instance.
\end{proof}

\begin{corollary}\label{pr:ROU}
Then the (fixed base) Diophantine problem for equations \eqref{eq:semenov} is $\NP$-complete.\qed
\end{corollary}
\begin{proof}
    $\NP$-hardness is proved above. To establish that the problem is in $\NP$, observe that if there is a solution, there must be one satisfying
    $$0 \leq x_1,\ldots,x_k < n,$$ where $n$ is the order of $\alpha$, and it is easily seen that $n$ is $O(\size^2(E))$. It is shown in Lemma \ref{le:check-sol} that such a solution may be verified in time polynomial in $\size(E)$.
\end{proof}
There exist polynomial-time algorithms (see, for instance \cite{Bradford-Davenport:1988}) to determine, based on its minimal polynomial, whether $\alpha\in \CQ$ is a root of unity (and if so, to find its order). Thus, it is valid to consider this case separately for the uniform Diophantine problem (when $\alpha$ is a part of the input), and we may consequently assume that $\alpha$ is \emph{not} a root of unity in Section \ref{se:semenov-np-equation} (where it is proved that the Diophantine problem for a single equation is in $\NP$).

\subsection{$\alpha$ is not a root of unity}

The following proposition is the main technical result of this section. In what follows, $\CU$ always denotes the set of roots of unity and
$\CQ^\ast = \CQ\setminus\{0\}$.

\begin{proposition}\label{pr:unique_solution_exp_eq}
Let $\alpha\in \CQ^*\setminus \CU$ and $d = \deg(\alpha)$.
For $s\in \MN$, define the positive integer
\begin{equation}
c(\alpha, s) = \left \lceil \frac{3(\ln 2 + 2\ln s)}{\kappa(d)}\right \rceil,
\end{equation}
where
$$
\kappa(d) = \begin{cases} \ln 2 &\mbox{for } \ d = 1,\\
\frac{2}{d \cdot (\ln(3d))^3} &\mbox{for } \ d \geq 2.
\end{cases}
$$
Then for any integers $0\le p_1<p_2<\dots<p_s$  satisfying $p_{i+1}-p_{i} \geq c(\alpha, s)$,
the equation
\begin{equation}\label{eq:exp_equality}
\alpha^{x_1}+\dots+\alpha^{x_s}=
\alpha^{p_1}+\dots+\alpha^{p_s}
\end{equation}
has (up to a permutation) the unique integer solution $x_i = p_i$.
\end{proposition}

For the proof, we make use of the following lemma,
proved by H. W. Lenstra, see \cite[Proposition 2.3]{Lenstra:1999}.
Note that Lenstra's original result is more general, as it applies to polynomials in $\CQ[x]$;
we only require the slightly weaker version for polynomials in $\MZ[x]$
stated below. For a polynomial $f(x) = \sum_{i=1}^h c_i x^i \in \MZ[x]$ define
$$
\height(f) = \max(|c_1|,\ldots,|c_h|).
$$

\begin{lemma}[Lenstra]\label{le:Lenstra}
Let $\alpha,d$ and $\kappa(d)$ be as
in Proposition \ref{pr:unique_solution_exp_eq}, and suppose that
a polynomial $P(x) = P_0(x) + x^r P_1(x) \in \MZ[x]$
contains $k$ monomials and satisfies
\begin{equation}\label{ineq:Lenstra}
r-\deg(P_0) > \frac{\ln (k-1) + \ln (\height(P))}{\kappa(d)}.
\end{equation}
Then $P(\alpha) = 0$ if and only if $P_0(\alpha)=0$ and $P_1(\alpha) = 0$.
\qed
\end{lemma}

\begin{proof}[Proof of Proposition \ref{pr:unique_solution_exp_eq}]
By way of contradiction, let $0\le p_1<p_2<\dots<p_s$  satisfy $p_{i+1}-p_{i} \geq c(\alpha, s)$, and suppose there is a solution $x_1\le \dots\le x_s\in\MZ$ to \eqref{eq:exp_equality} such that $\{x_1,\ldots,x_s\}\neq \{p_1,\ldots,p_s\}$. Multiplying by an appropriate power of $\alpha$ if necessary, we may assume that the $x_i$ are nonnegative. Eliminating terms that appear on both sides and reindexing, we obtain $x_1,\ldots,x_{s'}$ and $p_1,\ldots,p_{s'}$ such that $\label{eq:intersection}
\{x_1,\ldots,x_s\}\cap \{p_1,\ldots,p_s\} = \emptyset$ and $s' \leq s$. Hence, the polynomial
$$P(x) = \sum_{i=1}^s x^{p_i} - \sum_{i=1}^s x^{x_i}$$ has $k = 2s'$ monomials, with $k \leq 2s$, and satisfies $P(\alpha) = 0$. Since the coefficients of $P$ are bounded in absolute value by $s$ (this bound is realized if $s=s'$ and $x_1 = \cdots = x_s$), we have $\height(P) \leq s$. Hence, the right-hand side of the inequality in Lemma \ref{le:Lenstra} corresponding to $P$ is less than $$K = \frac{\ln 2 + 2\ln s}{\kappa(d)},$$
and we notice that $c(\alpha,s) = \lceil 3K\rceil \geq 3K$. Now suppose that there is some $t\in \{1,\ldots,s\}$ such that none of the $x_i$ are contained in the interval $[p_t - K, p_t + K]$. Applying Lemma \ref{le:Lenstra} twice, this implies that $\alpha^{p_t} = 0$, a contradiction. Thus, by a pigeonhole argument there is exactly one $x_i$ contained in each of the $s$ disjoint $K$-neighborhoods of $p_1,\ldots,p_s$. This means that $P(x)$ is a sum of $s$ polynomials of the form $B_i(x) = x^{p_i}-x^{x_i}$, such that $$\min(p_{i+1},x_{i+1}) - \max(p_i,x_i) > K.$$ Hence, Lemma \ref{le:Lenstra} shows that $B_1(\alpha) = 0$, implying that $\alpha$ is a root of unity, which contradicts our assumption.
\end{proof}

Fix $\alpha\in\CQ^\ast\setminus\CU$.
Below we prove $\NP$-hardness of the Diophantine problem
for \eqref{eq:semenov} by reducing the $3$-partition problem
to the decidability  of \eqref{eq:semenov}.
For a given multiset $S=\{a_1,\ldots,a_{3k}\}$ of $3k$ integers, define
$$
L = \frac{1}{k}\sum_{i=1}^{3k}a_i.
$$
The $3$-partition problem  (abbreviated as $\TPART$)
is the problem of deciding whether an integer multiset $S=\{a_1,\ldots,a_{3k}\}$, where $L/4 < a_i < L/2$, can be partitioned into $k$ triples, each of which sums to $L$.
This problem is known to be \emph{strongly} $\NP$-complete, which means that it remains $\NP$-complete even when the input is represented in unary. A thorough treatment of this problem may be found in \cite{Garey-Johnson:1979}.
Below we reduce an instance of $\TPART$ to an instance of the Diophantine problem for \eqref{eq:semenov}. Note that because of the restriction $L/4 < a_i < L/2$,
we may assume that $S$ contains only positive integers.

Let $S=\{a_1,\ldots,a_{3k}\}$ be an instance of
$\TPART$, with $L = \frac{1}{k}\sum_{i=1}^{3k} a_i$
the anticipated sum and $L/4 < a_i < L/2$ (we may assume that $L\in\MN$). Let $c = c(\alpha, Lk)$ be defined as in Proposition \ref{pr:unique_solution_exp_eq}, and define the numbers
\begin{align*}
q_y &= 1+\alpha^c+\alpha^{2c}+\dots+\alpha^{(y-1)c}
\ \ \ \ \ \ \mbox{ for }\ \  y\in\MN\\
r &= q_L\rb{1+\alpha^{2cL}+\alpha^{4cL}+\dots+\alpha^{2(k-1)cL}}
\end{align*}
and the equation
\begin{equation}\label{eq:3-part}
q_{a_1}\alpha^{x_1}+\cdots+q_{a_{3k}}\alpha^{x_{3k}} = r.
\end{equation}

\begin{proposition}\label{pr:3-part}
$S$ is a positive instance of $\TPART$
if and only if \eqref{eq:3-part} has a solution.
Furthermore, a solution $x_1,\ldots,x_{3k}$ for \eqref{eq:3-part},
if it exists, is unique up to a permutation and satisfies
\begin{equation}\label{3PART-reduction-bound}
 0\le x_1,\ldots,x_{3k}\le 2ckL.
\end{equation}

\end{proposition}

\begin{proof}
Suppose that $S$ is a positive instance of $\TPART$.
Reindexing the $a_i$ and $x_i$ if necessary, we may assume that $\sum_{j=1}^3 a_{3i+j} = L$ for $i = 0,1,\ldots, k-1$. It is now easily checked that
\begin{align*}
x_{3i+1} &= 2icL \\
x_{3i+2} &= c(2iL + a_{3i+1}) \\
x_{3i+3} &= c(2iL + a_{3i+1} + a_{3i+2})
\end{align*}
for $i = 0,1,\ldots,k-1$
satisfies \eqref{eq:3-part} and \eqref{3PART-reduction-bound}.

For the other direction, suppose that $x_1,\ldots,x_{3k}$ is a solution
of \eqref{eq:3-part}.
By construction, the left-hand side of (\ref{eq:3-part}) is
a sum of $Lk$ powers of $\alpha$, while
the right-hand side is a sum of $Lk$ distinct powers of $\alpha^c$. In particular, the sum on the right-hand side contains blocks of consecutive powers of $\alpha^c$, with gaps between $\alpha^{(2i-1)c(L-1)}$ and $\alpha^{2icL}$ for $i = 1,\ldots,k-1$. Proposition  \ref{pr:unique_solution_exp_eq} implies
that the left-hand side consists of the same distinct
powers of $\alpha$, and the proof follows from a careful comparison of the powers on each side. First of all, it is clear that we must have $x_{i_1} = 0$ for exactly one $x_{i_1}$,
and
$q_{a_{i_1}}\alpha^{x_{i_1}} =
1 + \alpha^c + \ldots + \alpha^{c(a_{i_1}-1)}$.
Since $ca_{i_1} < cL$ by assumption, the right-hand side of \eqref{eq:3-part} contains $\alpha^{ca_{i_1}}$ and so we must have $x_{i_2} = ca_{i_1}$ for some (unique) $x_{i_2}$.
Similarly, the highest degree term of
$q_{a_{i_1}}\alpha^{x_{i_1}} + q_{a_{i_2}}\alpha^{x_{i_2}}$
is $\alpha^{c(a_{i_1}+a_{i_2}-1)}$, and since $c(a_{i_1}+a_{i_2}) < cL$, we must have the next consecutive power of $\alpha^c$ on the left-hand side. Hence, there must be
$x_{i_3} = c(a_{i_1}+a_{i_2})$. Finally, the highest power
of $q_{a_{i_1}}\alpha^{x_{i_1}} + q_{a_{i_2}}\alpha^{x_{i_2}} + q_{a_{i_3}}\alpha^{x_{i_3}}$ is $\alpha^{c(a_{i_1}+a_{i_2} + a_{i_3}-1)}$, which must be the last term in the first block of consecutive powers of $\alpha^c$ (since for any $a_l$ we have $cL < c(a_{i_1}+a_{i_2}+a_{i_3} + a_l)$). This implies that $c(a_{i_1}+a_{i_2}+a_{i_3}) = cL$, and the next largest $x_i$ is equal to $2cL$. It is clear that this process may be continued to show that $S$ is a positive instance, and that $x_1,
\ldots,x_{3k}$ satisfies \eqref{3PART-reduction-bound}.
\end{proof}

\begin{corollary}\label{co:semenov-NP-hard}
For $\alpha\in\CQ\setminus\{0,1\}$, the Diophantine problem for equations \eqref{eq:semenov}
is $\NP$-hard.
\end{corollary}

\begin{proof}
If $\alpha\in\CQ\setminus\{0,1\}$ is not a root of unity, Proposition \ref{pr:3-part} provides a polynomial-time (Karp) reduction from $\TPART$
to the Diophantine problem for \eqref{eq:semenov}. The case where $\alpha\in \CU$ was proved in Proposition \ref{pr:ROU}.
\end{proof}

\section{Complexity upper bound for a single equation}
\label{se:semenov-np-equation}

In this section, we prove that the Diophantine problem for single algebraic equations with exponents is in $\NP$. The following lemma is required to establish a polynomial time procedure for checking a candidate solution. Note that the proof extends easily to the case of a system of equations; however, for convenience we state it for a single equation.

\begin{lemma}\label{le:check-sol}
Consider an equation $E$ of type \eqref{eq:semenov}. Let $\ovz = (z_1,\ldots,z_k)\in \MZ^k$, and set
\begin{align*}
M &= \max\{1,|z_1|,\ldots,|z_k|\}, \\
\CM(E,\ovz) &= \max \{\size(E),\log M \}.
\end{align*}
Then there is an algorithm that checks whether $x_i = z_i$ is a solution of $E$ in time polynomial in $\CM(E,\ovz)$.
\end{lemma}

\begin{proof}
Let $d = \deg(\alpha)$, so that each
$q_i$ is equal to $\sum_{j=0}^{d-1}r_{ij}\alpha^{j}$
for some $r_{ij} \in \MQ$.
Multiplying $E$ by the product of all denominators,
which is less than $2^{\size(E)}$,
we can ensure that each $r_{ij}\in \MZ$. This can be done in time polynomial in $\size(E)$, and the new equation $E'$ satisfies $\size(E') = O(\size^2(E))$.

Set $\underline{m} = \min\{z_1,\ldots,z_k\}$ and $\overline{m}= \max\{z_1, \ldots,z_k\}$, and let $z_i' = z_i-\underline{m}$ (so that each $z_i'$ is non-negative and $\min z_i' = 0$); note that $z_1',\ldots,z_k'$ may be computed in time polynomial in $\log M$. Define $f_i(x) = \sum_{j=0}^{d-1}r_{ij} x^{j}$ for $i=0,\dots,k$, and
$$
F(x) = f_1(x)x^{z_1'}+\cdots+f_k(x)x^{z_k'} - f_0(x)\ \in\ \MZ[x].
$$
To compute the coefficients of $F(x)$, we must perform comparisons and
addition on $d(k+1)$ terms, with coefficients and exponents encoded
as binary numbers of length bounded by $\size(E)$ and
$\log d + \log(\overline{m}-\underline{m}+1)$, respectively.
That can be done in time polynomial in $\CM(E,\ovz)$, because
$d,k$ are bounded by $\size(E)$. In this way, we end up with the \emph{sparse representation} of $F$, i.e.,
the list of pairs $(c_n,n)$, where $c_n$ is the coefficient of $x^n$ in $F(x)$ and $c_n \neq 0$, with $c_n$ and $n$ given in binary. Moreover, it is encoded in space that is polynomial in $\CM(E,\ovz)$.

By construction, $z_1',\dots,z_k'$
satisfy \eqref{eq:semenov} if and only if $\alpha$ is a zero of $F(x)$.
By \cite[Theorem 2.1]{Lohrey:2020}, the problem to decide
if $F(\alpha)=0$ (where the input consists of $F(x)$ and the minimal polynomial of $\alpha$, encoded as above) belongs to $\bf{TC}^0$. In particular, it is polynomial-time decidable, so the result follows.
\end{proof}
Let $\alpha\in\CQ^*\setminus \CU$, and let $d = \deg(\alpha)$. Consider, instead of (\ref{eq:semenov}), the equation
\begin{equation}\label{eq:semenov3}
q_1\alpha^{x_1}+\dots+q_k\alpha^{x_k}= 0
\end{equation}
with coefficients $q_1,\ldots,q_k\in\MZ[\alpha]\setminus \{0\}$, and each coefficient given as
$$
q_i = r_{i0}+r_{i1}\alpha + \cdots + r_{i(d-1)}\alpha^{d-1}, \ \ r_{ij}\in \MZ.
$$
We call \eqref{eq:semenov3} a \emph{homogeneous} equation with exponents.
The Diophantine problem for \eqref{eq:semenov3}
is easily seen to be polynomial-time equivalent (in terms of $\size(E)$) to
the Diophantine problem for \eqref{eq:semenov}.

\subsection{Block structure of a solution}

Let $x_1,\ldots,x_k\in\MZ$ be a solution for \eqref{eq:semenov3}.
Denote the vector of the solution $(x_1,\ldots,x_k)\in\MZ^k$ by $\ovx$.
A nonempty set $I\subseteq \{1,\ldots,k\}$ is called a \emph{block}
for $\ovx$ if the following conditions hold:
\begin{itemize}
\item[(B1)]
$\sum_{i\in I} q_i \alpha^{x_i}=0$;
\item[(B2)]
$I$ does not have a nonempty proper subset satisfying (B1).
\end{itemize}
For a given solution $\ovx$, the set of indices $\{1,\ldots,k\}$ can be represented
as a disjoint union $\bigsqcup_{j=1}^m I_j$ of blocks,
perhaps in more than one way.
Such a collection $\CI = \{I_1,\ldots,I_m\}$ is called a \emph{block structure} for a solution $\ovx$.
For a block $I$, define
\begin{itemize}
\item
$\Delta_I=(\delta_1,\ldots,\delta_k)\in\MZ^k$, where
$$
\delta_i=
\begin{cases}
1&\mbox{if }i\in I\\
0&\mbox{if }i\notin I
\end{cases},
$$
\item
$\ovx_I=\Set{x_i}{i\in I}\subseteq \MZ$.
\item
$\spn_I(\ovx) = \max (\ovx_I)-\min (\ovx_I)$, called the \emph{span} of $I$.
\end{itemize}
The next lemma follows
immediately from the definition of $\Delta_I$.

\begin{lemma}\label{le:shift-blocks}
If $\{I_1,\ldots,I_m\}$ is a block structure for
a solution $\ovx$ for \eqref{eq:semenov3}, then
$$
\ovx+b_1\Delta_{I_1}+\dots+b_m\Delta_{I_m}
$$
is a solution for \eqref{eq:semenov3} for any $b_1,\ldots,b_m\in\MZ$.
\qed
\end{lemma}

\subsection{Gap and maximum span of an equation}

Consider an equation $E$ of type \eqref{eq:semenov3}.
Define $\gap(E)$ to be the least $n\in\MN$ such that for every solution
$x_1,\dots,x_k \in \MZ$  for \eqref{eq:semenov3}
and for every partition $\{1,\dots,k\} = S_1\sqcup S_2$
the following holds:
\begin{equation}\label{eq:gap-def}
\min_{i\in S_2}x_i  - \max_{i\in S_1}x_i >n
\ \ \Rightarrow\ \ S_1\mbox{ and } S_2 \mbox{ are unions of blocks.}
\end{equation}
If $E$ has no solutions, then set $\gap(E)=0$. The following lemma shows that $\gap(E)$ is well-defined.

\begin{lemma}\label{le:gap-bound}
For equations of type \eqref{eq:semenov3}, $\gap(E)=O(\size^3(E))$. If the base $\alpha$ is fixed, or if $\alpha$ is restricted to rational values, then $\gap(E)=O(\size(E))$.
\end{lemma}

\begin{proof}
If $E$ has no solutions, then $\gap(E)=0$ and
the statement holds.
Let $\ovx=(x_1,\dots,x_k)\in\MZ^k$ be a solution for \eqref{eq:semenov3}
and $\CI$ a block structure for $\ovx$. As in the proof of Lemma \ref{le:check-sol}, construct the
polynomial $F(x)\in\MZ[x]$ corresponding to $\ovx$. The number of monomials in $F$ is not greater than $dk$
and
$$
\height(F)\leq  \sum_{i=1}^k\sum_{j=0}^{d-1}|r_{ij}|,
$$
which gives the following upper bound
on the ``gap bound'' (i.e. the right-hand side of inequality \eqref{ineq:Lenstra}) of Lemma \ref{le:Lenstra} for $F(x)$:
\begin{align*}
&\tfrac{d\ln^3(3d)}{2}
\bigg(\ln(kd-1)+\ln \bigg(\sum_{i=1}^k\sum_{j=0}^{d-1}|r_{ij}|\bigg)\bigg)\\
\le&
\tfrac{d\ln^3(3d)}{2}
\bigg(\ln(k)+\ln(d)+\ln\bigg(\prod_{i=1}^k\prod_{j=0}^{d-1}(|r_{ij}|+1)\bigg)\bigg)\\
\le&
\underbrace{\tfrac{d\ln^3(3d)}{2}}_{O(\size^2(E))}
\underbrace{\bigg(\ln(k)+\ln(d)+\sum_{i=1}^k\sum_{j=0}^{d-1}\ln(|r_{ij}|+1)\bigg)}_{O(\size(E))}.
\end{align*}
By construction, each monomial in $F(x)$ is of the form $cx^{x_{i}+j}$, where $0\le j<d$.
Hence, for a partition $\{1,\dots,k\} = S_1\sqcup S_2$
satisfying
\begin{equation}\label{eq:gap-ineq}
\min_{i\in S_2}x_i  - \max_{i\in S_1}x_i >
d+\tfrac{d\ln^3(3d)}{2}
\bigg(\ln(k)+\ln(d)+\sum_{i=1}^k\sum_{j=0}^{d-1}\ln(|r_{ij}|+1)\bigg),
\end{equation}
we may write $F(x)$ as a sum $F_1(x)+F_2(x)$ which satisfies the assumptions of Lemma \ref{le:Lenstra}, where $F_1$ contains the terms of degree less than $d+ \max_{i\in S_1}x_i$ and $F_2$ contains the terms of degree at least $\min_{i\in S_2}x_i$. Therefore, Lemma \ref{le:Lenstra} implies that $F_1(\alpha) = F_2(\alpha) = 0$, so each block in $\CI$ belongs either to $S_1$ or to $S_2$.
The right-hand side of \eqref{eq:gap-ineq} is $O(\size^3(E))$, which completes the proof in the general case (where $\alpha$ is part of the input). In the case where $\alpha$ (hence, $d$) is fixed, the right-hand side of \eqref{eq:gap-ineq} is $O(\size(E))$. If $\alpha\in \MQ$, then $\frac{1}{\ln 2}$ replaces $\frac{d\ln^3(3d)}{2}$ in \eqref{eq:gap-ineq}, so we obtain $O(\size(E))$ in this case as well.
\end{proof}

For an equation $E$ of type \eqref{eq:semenov3} define
$$
\mspn(E)=
\max \Set{\spn_I(\ovx)}{I \mbox{ is a block in a solution $\ovx$ for $E$}},
$$
called the \emph{maximum span} of $E$. If $E$ has no solutions, then
$\mspn(E)=0$. The following lemma shows that the maximum span is well-defined.
\begin{lemma}\label{le:max-span}
$\mspn(E) = O(\size^4(E))$
for every equation $E$ of type \eqref{eq:semenov3}. If the base $\alpha$ is fixed, then
$\mspn(E) = O(\size^2(E)),$ and if $\alpha$ is restricted to rational values then $\mspn(E) = O(\size(E)).$
\end{lemma}

\begin{proof}
If $I$ is a block in a solution $\ovx = (x_1,\ldots,x_k)$ of an equation $E$ of type \eqref{eq:semenov3}, then it is clear from the definitions that $|x_{i_1}-x_{i_2}|\leq \gap(E)$ for any $i_1,i_2\in I$. Thus, we obtain
\begin{equation}\label{eq:max-span}
\spn_I(\ovx) \leq |I|\gap(E)\leq (k-1)\gap(E).
\end{equation}
Therefore,
$$
\mspn(E)\leq (k-1)\gap(E),
$$
whence the first and second statements follow immediately from Lemma \ref{le:gap-bound} (together with the fact that $k \leq \size(E)$). For the case where $\alpha$ is restricted to rational values, we defer the proof to Section \ref{se:rational-base}.
\end{proof}

\subsection{An $\NP$-certificate for \eqref{eq:semenov3}}
\label{se:certificate_for_semenov3}
In this section, we establish a polynomial upper bound on a minimal solution of \eqref{eq:semenov3}, and demonstrate that such a solution may be checked in polynomial time.
\begin{theorem}\label{th:Semenov-NP}
If an equation $E$ of type \eqref{eq:semenov3} has a solution,
then there is a solution $(x_1,\ldots,x_k)\in\MZ^k$ satisfying
$$0 \leq x_1,\ldots,x_k \leq \mspn(E).$$
\end{theorem}

\begin{proof}
Consider any solution $\ovx=(x_1,\ldots,x_k)\in\MZ^k$ and
a block structure $\CI=\{I_1,\ldots,I_m\}$ for $\ovx$.
By Lemma \ref{le:shift-blocks}
$$
(x_1',\ldots,x_k') = \ovx-\sum_{j=1}^m \min( \ovx_{I_j})\cdot \Delta_{I_j}
$$
is also a solution. By construction, it satisfies
$0 \leq x_1',\ldots,x_k' \leq \mspn(E)$,
as claimed.
\end{proof}

\begin{corollary}\label{co:Semenov-NP}
The Diophantine problem for equations \eqref{eq:semenov3}
belongs to $\NP$.
\end{corollary}

\begin{proof}
By Lemma \ref{le:check-sol}, a solution for \eqref{eq:semenov3}
satisfying the conclusion of Theorem \ref{th:Semenov-NP}
constitutes an $\NP$-certificate.
\end{proof}

\begin{corollary}\label{co:Semenov-NP-complete}
The (uniform or fixed base) Diophantine problem for
\eqref{eq:semenov3} is $\NP$-complete for $\alpha \in\CQ\setminus\{0,1\}$.
\end{corollary}

\begin{proof}
This follows from Proposition \ref{pr:ROU} and Corollaries \ref{co:semenov-NP-hard} and \ref{co:Semenov-NP}.
\end{proof}

\subsection{An improved bound on maximum span for $\alpha\in \MQ$}
\label{se:rational-base}

It turns out that in the special case where $\alpha\in \MQ$,
one can obtain the following improved bound on $\mspn(E)$.

\begin{theorem}\label{th:sol-bound-rational}
Let $E$ be an equation of type \eqref{eq:semenov3}, such that $\alpha\in \MQ \setminus \{-1,0,1\}$. Then
\begin{enumerate}[(a)]
    \item\label{st:mspn-bound} $\mspn(E)\leq \abs{\log_{|\alpha|} (|q_i|+1)}$
    \item\label{st:mspn-lin} $\mspn(E) = O(\size(E))$ for fixed $\alpha$.
\end{enumerate}
\end{theorem}

This is very useful in certain applications in which the base $\alpha$ is fixed (e.g., the authors make use of this result in \cite{Mandel-Ushakov:2022}). Note that statement \eqref{st:mspn-lin} of Theorem \ref{th:sol-bound-rational} is an immediate consequence of \eqref{st:mspn-bound}, which we prove in the remainder of this section. However, we do not refer to this material anywhere else in the present work. For convenience, we sometimes assume that a solution $x_1,\ldots,x_k$ has a block structure consisting of a single block, and that $x_1 \geq x_2 \geq \cdots \geq x_k=0$. The following technical lemma is required.
\begin{lemma}\label{le:aux1}
Let $k\ge 2$, $\alpha > 1$, and $q_1,\ldots,q_k\in\MZ\setminus\{0\}$.
The maximum value of the sum $\delta_1+\cdots+\delta_{k-1}$,
for $\delta_1,\ldots,\delta_{k-1}\in\MR$ constrained by
\begin{equation}\label{constraints}
\resizebox{0.91\hsize}{!}{%
$
\begin{array}{l}
0 \leq \delta_i \leq \log_{\alpha} \rb{|q_{i+1}|+|q_{i+2}|\alpha^{-\delta_{i+1}}+\cdots +|q_k|\alpha^{-(\delta_{i+1}+\cdots+\delta_{k-1})}}\ \mbox{ for }i = 1,\ldots,k-2,\\
0 \leq \delta_{k-1} \leq \log_{\alpha} (|q_k|),
\end{array}
$
}
\end{equation}
is attained at
$\delta_1=\log_{\alpha}(|q_2|+1),
\ldots,
\delta_{k-2}=\log_{\alpha}(|q_{k-1}|+1),
\delta_{k-1}=\log_{\alpha}(|q_{k}|)$.
\end{lemma}

\begin{proof}
Since \eqref{constraints} defines a non-empty compact set,
the sum attains the maximum value at some point $(\delta_1,\ldots,\delta_{k-1})$.
Fix $\delta_3,\ldots,\delta_{k-1}$, define
$$
q^\ast=
|q_{3}|+|q_{4}|\alpha^{-\delta_{3}}+\ldots+|q_k|\alpha^{-(\delta_{3}+\cdots+\delta_{k-1})},
$$
and notice that $\delta_1,\delta_2$
must maximize $\delta_1+\delta_2$ while satisfying
\begin{align*}
0\le\delta_1\le&\log_{\alpha}\rb{|q_2|+|q^\ast|\alpha^{-\delta_2}},\\
0\le\delta_2\le&\log_{\alpha}\rb{|q^\ast|}.
\end{align*}
Obviously, $\delta_1=\log_{\alpha}\rb{|q_2|+|q^\ast|\alpha^{-\delta_2}}$.
Thus, the value of $\delta_2$ must maximize the value of the function
\begin{align*}
&f(\delta_{2})=\delta_{2}+\log_{\alpha}\rb{|q_{2}|+|q^\ast|\alpha^{-\delta_{2}}}
\end{align*}
for $\delta_{2}\in[0,\log_{\alpha}(|q^\ast|)]$.
It is easily seen that $f$ is monotonically increasing, so the maximum is
attained at $\delta_{2}=\log_{\alpha}(|q^\ast|)$,
and $\delta_{1} = \log_{\alpha}\rb{|q_{2}|+1}$.

Once the optimal value of $\delta_1$ is found, we can eliminate $\delta_1$
from the sum and remove the bounds on $\delta_1$ from
\eqref{constraints} (notice that $\delta_1$ is not involved in the other bounds).
That produces an optimization problem of the same type with $k-1$ variables.
Hence, the result follows by induction on $k$.
\end{proof}

\begin{lemma}\label{aux2}
Suppose that $1<\alpha\in\MQ$ and
$\ovx = (x_1,\ldots,x_k)$ is a solution for an equation $E$ of type \eqref{eq:semenov3}
with a single block $I=\{1,\ldots,k\}$
satisfying $x_1 \ge x_2 \ge \dots \ge x_k=0$.
Define $\delta_i=x_i-x_{i+1}$ for $i=1,\ldots,k-1$.
Then we have
\begin{itemize}
\item[(a)]
$0\le \delta_i\le \log_{\alpha} \rb{|q_{i+1}|+|q_{i+2}|\alpha^{-\delta_{i+1}}+\cdots +|q_k|\alpha^{-(\delta_{i+1}+\cdots+\delta_{k-1})}}$
\item[(b)]
$\spn_I(\ovx) = x_1\le \sum_{i=1}^{k} \log_{\alpha}(|q_i|+1)$.
\end{itemize}
\end{lemma}

\begin{proof}
Let $\alpha=\tfrac{c}{d}$, where $\gcd(c,d)=1$, $c\ge 2$ and $c>d$.
Multiplying \eqref{eq:semenov3} by $d^{x_1}$ we get
\begin{equation}\label{eq:semenov_c}
q_1 c^{x_1} + q_2 c^{x_2}d^{x_1-x_2} +\cdots+q_{k-1} c^{x_{k-1}}d^{x_1-x_{k-1}}
+q_k c^{x_k}d^{x_1-x_k} = 0.
\end{equation}
Taking \eqref{eq:semenov_c} modulo $c^{x_{k-1}}$ we get
\begin{align*}
q_k c^{x_k}d^{x_1-x_k} \equiv 0 \bmod c^{x_{k-1}}
&\ \ \Rightarrow\ \
c^{x_{k-1}-x_k}\mid q_k\\
&\ \ \Rightarrow\ \
\delta_{k-1} = x_{k-1}-x_k \leq \log_c(|q_k|)\\
&\ \ \Rightarrow\ \
\delta_{k-1} \leq \log_\alpha(|q_k|),
\end{align*}
which proves (a) for $i=k-1$.
Writing (\ref{eq:semenov_c}) as
\begin{equation}\label{eq:semenov_c2}
q_1 c^{x_1} + q_2 c^{x_2}d^{x_1-x_2}
+\cdots+
\underbrace{\rb{q_{k-1} + q_k\rb{\tfrac{d}{c}}^{\delta_{k-1}}}}_{\ \ \ {}^\parallel_ {q_{k-1}'}} c^{x_{k-1}}d^{x_1-x_{k-1}}=0,
\end{equation}
we obtain an expression of the same form as (\ref{eq:semenov_c}), with one fewer term and a new rightmost coefficient $q_{k-1}'\in\MZ$ that satisfies the following:
\begin{itemize}
\item
$q_{k-1}'=0\ \ \Leftrightarrow\ \ k=2$
(because, by our assumption, $\ovx$ has a single block);
\item
$|q_{k-1}'|\le|q_{k-1}|+|q_{k}|$.
\end{itemize}
Assuming $k > 2$, we may apply the same argument to \eqref{eq:semenov_c2}. That is, take
\eqref{eq:semenov_c2} modulo $c^{x_{k-2}}$, and follow the same steps as above to obtain $\delta_{k-2} \leq \log_\alpha(|q_{k-1}'|)$, proving (a) for $i=k-2$. This process can be continued to yield,
at each iteration, a new rightmost (non-trivial unless $i=1$) coefficient
$$
q_i' =
q_i+q_{i+1}'\alpha^{-\delta_i} =
q_{i}+q_{i+1}\alpha^{-\delta_{i}}+\cdots +q_k\alpha^{-(\delta_{i}+\cdots+\delta_{k-1})}
$$
satisfying (a) for each $i = 1,\ldots,k-1$.

To prove (b), notice that $x_1$ can be expressed as
$$
x_1=\delta_1+\dots+\delta_{k-1}
$$
and by (a) each $\delta_i$ satisfies the constraints of Lemma \ref{le:aux1}.
Hence, it follows from Lemma \ref{le:aux1} that
$x_1\le \sum_{i=1}^{k} \log_\alpha(|q_i|+1)$.
\end{proof}
Lemma \ref{aux2} was proved under the assumption that $\alpha>1$.
Below we show that it holds with minor modifications for any
$\alpha\in\MQ\setminus \{-1,0,1\}$.
\begin{proposition}\label{pr:block-length}
Suppose that $\alpha\in\MQ\setminus \{-1,0,1\}$ and
$\ovx = (x_1,\ldots,x_k)$ is a solution for an equation $E$ of type \eqref{eq:semenov3}
with a single block $I=\{1,\ldots,k\}$
satisfying $x_1 \ge x_2 \ge \dots \ge x_k=0$.
Define $\delta_i=x_i-x_{i+1}$ for $i=1,\ldots,k-1$.
Then we have
$$\spn_I(\ovx) = x_1\le \sum_{i=1}^{k} \abs{\log_{|\alpha|}(|q_i|+1)}.$$
\end{proposition}

\begin{proof} We consider four cases.
\begin{case}\label{ca:one}
Suppose that $\alpha>1$.
Then the statement follows from Lemma \ref{aux2}.
\end{case}
\begin{case}
The case when $0<\alpha<1$
is reduced to Case \ref{ca:one} by replacing $x_i$ with $-x_i$ and $\alpha$ with $1/\alpha$.
This transformation changes the order of $x_i$'s and $\delta_i$'s, which also
modifies the inequalities in item (a) of Lemma \ref{aux2} in the following way:
$$
\delta_1\le \log_{1/\alpha}(|q_1|),\ \
\delta_2\le \log_{1/\alpha}(|q_2|+|q_1|(\tfrac{1}{\alpha})^{-\delta_1}),\ \ \mbox{etc.}
$$
Since the obtained inequalities are of the same type,
Lemma \ref{le:aux1} remains applicable and gives the claimed bound
$$
x_1\le \sum_{i=1}^{k} \log_{1/\alpha}(|q_i|+1) =
\sum_{i=1}^{k} \abs{\log_{\alpha}(|q_i|+1)}.
$$
\end{case}
\begin{case}\label{ca:three}
Suppose that $\alpha<-1$. It is easy to see that
Lemma \ref{aux2} holds in that case with minor modifications
(with $\alpha$ replaced with $|\alpha|$).
\end{case}
\begin{case}
The case when $-1<\alpha<0$ is reduced to Case \ref{ca:three} by replacing $x_i$ with $-x_i$ and $\alpha$ with $1/\alpha$.
\qedhere
\end{case}
\end{proof}

Theorem \ref{th:sol-bound-rational} now follows from Proposition \ref{pr:block-length}.

\section{Complexity upper bound for a system of homogeneous equations}
\label{se:semenov-np-homo-system}

Let $\alpha_1,\ldots,\alpha_s\in\CQ^*\setminus\CU$
and $d_i = \deg(\alpha_i)$. Consider a system of equations
\begin{equation}\label{eq:semenov4}
\arraycolsep=2pt
\left\{
\begin{array}{cl}
q_{11}\alpha_1^{x_1}+\dots+q_{1k}\alpha_1^{x_k}&= 0\\
\vdots&\\
q_{s1}\alpha_s^{x_1}+\dots+q_{sk}\alpha_s^{x_k}&= 0\\
\end{array}
\right.
\end{equation}
with coefficients $q_{ij}\in\MZ[\alpha_i]$ and each coefficient given as $q_{ij} = r_{ij}^0 + r_{ij}^1 \alpha_i + \cdots + r_{ij}^{d_i-1}\alpha^{d_i-1}$. Let us also require that at least one of $q_{1j}, q_{2j}, \dots , q_{sj}$ is nonzero for each $j$. Let $\ovx=(x_1,\ldots,x_k)\in\MZ^k$ be a solution for \eqref{eq:semenov4}.
A nonempty $J\subseteq \{1,\ldots,k\}$ is called a \emph{cluster}
for the solution $\ovx$ if the following conditions hold:
\begin{itemize}
\item[(C1)]
$\sum_{j\in J} q_{ij} \alpha_i^{x_j}=0$ for every $i=1,\ldots,s$;
\item[(C2)]
$J$ does not have a nonempty proper subset $J$ satisfying (C1).
\end{itemize}
The set of indices $\{1,\ldots,k\}$ can be represented as a union of disjoint clusters,
perhaps in more than one way; a choice $\CJ = \{J_1,\ldots,J_m\}$
of one such union is called a \emph{cluster structure} for a solution $\ovx$. For a cluster $J$, we define $\ovx_J$, $\spn_J(\ovx)$ and $\Delta_J$ analogously to the case of a block. The next lemma follows immediately from the definitions.

\begin{lemma}\label{le:shift-blocks2}
If $\ovx$ is a solution for \eqref{eq:semenov4}
with cluster structure $J_1,\ldots,J_m$, then
$$
\ovx+\beta_1\Delta_{J_1}+\dots+\beta_m\Delta_{J_m}
$$
is a solution for \eqref{eq:semenov4}
with the same cluster structure
for any $\beta_1,\ldots,\beta_m\in\MZ$.
\qed
\end{lemma}

\subsection{Max span of a system}
For a system $E$ of type \eqref{eq:semenov4}, the definition of the maximum span is extended in the natural way, as follows:
$$
\mspn(E)=
\max \Set{\spn_J(\ovx)}{J \mbox{ is a cluster in a solution $\ovx$ for $E$}}.
$$
If $E$ has no solutions, then $\mspn(E)=0$.
We obtain an upper bound on $\mspn(E)$ below.

A collection $A_1,\ldots,A_m$ (where $m\ge 2$)
of finite nonempty subsets of $\MZ$
is \emph{non-separable} if the union of closed intervals
$$
\bigcup_{i=1}^m\ [\min(A_i),\max(A_i)]
$$
is an interval. Otherwise it is separable.
The following lemma follows easily from this definition.

\begin{lemma}\label{le:non-separable-length}
If finite nonempty subsets $A_1,\ldots,A_m\subseteq\MZ$
are non-separable, then the following inequality holds:
$$
\max\big(\bigcup_{i=1}^m A_i\big) - \min\big(\bigcup_{i=1}^m A_i\big)
\ \le\
\sum_{i=1}^m \big(\max(A_i)-\min(A_i)\big).
$$
\end{lemma}

Let $\ovx$ be a solution for \eqref{eq:semenov4} and
$\CJ$ a cluster structure for $\ovx$.
Then $\ovx$ is a solution for each individual equation in \eqref{eq:semenov4}.
Hence, each equation in \eqref{eq:semenov4} has
a block structure $\CI_i=\{I_{i1},\ldots,I_{i \, m_i}\}$ for $\ovx$.
We say that $\CI_1,\ldots,\CI_s$ are \emph{compatible} with $\CJ$ if
for any cluster $J\in\CJ$ and any block $I_{ij}\in\CI_i$ we have
$$
I_{ij}\subseteq J \mbox{ or }I_{ij}\cap J=\emptyset.
$$
For any cluster structure $\CJ$, (C1) obviously implies the existence
of a compatible block structure $\CI_i$
for each individual equation in \eqref{eq:semenov4}, and the identity \begin{equation}\label{eq:cluster-union}
J=\bigcup_{I_{ij}\subseteq J}I_{ij}
\end{equation}
holds for every $J\in\CJ$ because every index in $\{1,\ldots,k\}$ belongs to at least one block $I_{ij}$ (since we assume that there is at least one nonzero $q_{ij}$ corresponding to each $x_i$).

\begin{lemma}\label{le:non-separable-blocks}
If $\CJ$ and $\CI_1,\ldots,\CI_s$ are compatible,
then for every $J\in\CJ$, the collection of sets
$\Set{\ovx_{I_{ij}}}{I_{ij}\subseteq J}$
is non-separable.
\end{lemma}

\begin{proof}
Separability contradicts the property (C2) of $J$.
\end{proof}

\begin{proposition}\label{pr:cluster-length-bound}
Let $E$ be a system of type \eqref{eq:semenov4}, and let $E_i$ denote the $i$th equation in $E$. Then we have
$$\mspn(E) \leq sk\max_{1\leq i \leq s}(\mspn(E_i)).$$
\end{proposition}

\begin{proof} Let $\CJ$ be a cluster structure for a solution $\ovx$ of $E$, and let $J \in \CJ$. We have
\begin{align*}
\max(\ovx_{J})-\min(\ovx_{J})
&=\max\Big(\bigcup_{I_{ij}\subseteq J} \ovx_{I_{ij}}\Big)-\min\Big(\bigcup_{I_{ij}\subseteq J} \ovx_{I_{ij}}\Big)
&&\mbox{(the identity \eqref{eq:cluster-union})}\\
&\le
\sum_{I_{ij}\subseteq J}
\spn_{I_{ij}}(\ovx)
&&\mbox{(Lemmas \ref{le:non-separable-length} and \ref{le:non-separable-blocks}})\\
&\le
\sum_{I_{ij}\subseteq J}
\mspn(E_i)
&&\\
&\le
sk\max_{1\leq i \leq s}(\mspn(E_i)). &&
\end{align*}
\end{proof}
\begin{corollary}\label{co:mspan-poly}
    For systems $E$ of type \eqref{eq:semenov4}, $\mspn(E)$ is bounded by a polynomial in $\size(E)$.
\end{corollary}
\begin{proof}
    Follows immediately from Proposition \ref{pr:cluster-length-bound} and Lemma \ref{le:max-span}.
\end{proof}

\subsection{An $\NP$-certificate for \eqref{eq:semenov4}}
\label{se:certificate_for_semenov4}

\begin{theorem}\label{th:Semenov-NP2}
If a system $E$ of type \eqref{eq:semenov4} has a solution,
then it has a solution $x_1,\ldots,x_k\in\MZ$ satisfying
\begin{equation}\label{eq:Semenov-bound2}
0
\ \le\
x_1,\ldots,x_k
\ \le \mspn(E).
\end{equation}
\end{theorem}

\begin{proof}
Same as the proof of Theorem \ref{th:Semenov-NP},
using Lemma \ref{le:shift-blocks2} instead of Lemma \ref{le:shift-blocks}.
\end{proof}
Finally, from Lemma \ref{le:check-sol}, Corollary \ref{co:mspan-poly} and Theorem \ref{eq:Semenov-bound2}, we obtain the following.
\begin{corollary}\label{co:Semenov-NP2}
The Diophantine problem for systems of type \eqref{eq:semenov4}
belongs to $\NP$.
\end{corollary}

\section{Complexity upper bound for general systems of equations}
\label{se:semenov-np-non-homo-system}
\subsection{Non-homogeneous equations}
Let $\alpha_1,\ldots,\alpha_s\in\CQ^*\setminus\CU$, and consider a system of non-homogeneous equations
\begin{equation}\label{eq:semenov5}
\arraycolsep=2pt
\left\{
\begin{array}{cl}
q_{11}\alpha_1^{x_1}+\dots+q_{1k}\alpha_1^{x_k}&= q_{10},\\
\vdots&\\
q_{s1}\alpha_s^{x_1}+\dots+q_{sk}\alpha_s^{x_k}&= q_{s0},\\
\end{array}
\right.
\end{equation}
with $q_{ij}\in\MZ[\alpha_i]$. It is easy to see that a system $E$ of this form has a solution
if and only if the homogeneous system $E'$
\begin{equation}\label{eq:semenov6}
\arraycolsep=2pt
\left\{
\begin{array}{cl}
q_{11}\alpha_1^{x_1}+\dots+q_{1k}\alpha_1^{x_k}-q_{10}\alpha_1^{x_0}&=0\\
\vdots\ \ \ \ \ \ \ \ \ \ \ \ \ &\\
q_{s1}\alpha_s^{x_1}+\dots+q_{sk}\alpha_s^{x_k}-q_{s0}\alpha_s^{x_0}&=0\\
\end{array}
\right.
\end{equation}
has a solution. Hence, an $\NP$-certificate for $E'$
can be used as an $\NP$-certificate for $E$. Since $\size(E') = \size(E)$, it follows from Corollary \ref{co:Semenov-NP2} that the Diophantine problem for non-homogeneous systems of type \eqref{eq:semenov5} belongs to $\NP$.

\subsection{Systems of equations with roots of unity}\label{se:gen-sys}

Finally, we consider the most general systems of equations with exponents, where
$\alpha_1,\ldots,\alpha_s$ are allowed to be roots of unity.
As above, the non-homogeneous case can be reduced to the homogeneous one, so in fact we need only consider homogeneous systems. Recall that there is a polynomial-time algorithm (see \cite{Bradford-Davenport:1988}) that takes as input the minimal polynomial of $\alpha_i$ and determines whether $\alpha_i$ is a root of unity. If it is a root of unity, this algorithm also provides the order of $\alpha_i$ in the group $\CU$. Hence, we consider systems of the form
\begin{equation}\label{eq:semenov8}
\arraycolsep=2pt
\left\{
\begin{array}{cl}
q_{11}\alpha_1^{x_1}+\dots+q_{1k}\alpha_1^{x_k}&= 0\\
\vdots&\\
q_{t1}\alpha_t^{x_1}+\dots+q_{tk}\alpha_t^{x_k}&= 0\\
\vdots&\\
q_{s1}\alpha_s^{x_1}+\dots+q_{sk}\alpha_s^{x_k}&= 0\\
\end{array}
\right.
\end{equation}
with $\alpha_1,\ldots,\alpha_t \not\in \CU$ and $\alpha_{t+1},\ldots,\alpha_s\in \CU$ for some $t\in \{0,\ldots,s\}$, and $q_{ij}\in \MZ[\alpha_i]$. We may assume that $n_i$ is the order of $\alpha_i$ in $\CU$ for $t < i \leq s$. From the fact that $\deg(\alpha_i) = \varphi(n_i)$ (where $\varphi$ denotes Euler's totient function), and using the well-known lower bound $\varphi(n) \geq \sqrt{\frac{n}{2}}$, we obtain the following bound on $n_i$.
\begin{equation}\label{eq:n-bound}
n_i < 2(\deg(\alpha_i))^2 < 2\size^2(E).
\end{equation}
We also assume that every unknown $x_i$ is non-trivially involved in $E$, i.e. for every $j\in \{1,\ldots,k\}$ there exists $i\in\{1,\ldots,s\}$ such that $q_{ij}\ne 0$.

First, assume that $E$ is an equation of type \eqref{eq:semenov8} in which all of
the bases are roots of unity, and let $N = \lcm(n_1,\ldots,n_s)$.
Clearly, if $\ovx = (x_1,\ldots,x_k)$ is a solution to $E$,
and $x_i\equiv x_i'\bmod N$, then
$\ovx'=(x_1',\ldots,x_k')$ is another solution to $E$.
Hence, if $E$ has a solution, then it has a solution
$\ovx = (x_1,\ldots,x_k)$ such that $0 \leq x_1,\ldots,x_k < N$.
From \eqref{eq:n-bound} we have
\begin{equation}\label{eq:N-exp-bound}
N < 2^k\size^{2k}(E),
\end{equation}
so that
\begin{equation}\label{eq:N-bound}
  \log N = O(2k \log(\size(E))) = O(\size^2(E)).
\end{equation}
By Lemma \ref{le:check-sol}, $\ovx$ is an $\NP$-certificate for decidability of $E$.

Now suppose that at least one, but not all, of the bases are roots of unity, and consider the subsystem $E_{\leq t}$ of the first $t$ equations (i.e. the equations where $\alpha_i \not \in \CU$). Let $\ovx = (x_1,\ldots,x_k)$ be a solution to $E$. In particular, $\ovx$ is a solution to $E_{\leq t}$, and we may choose a cluster structure $\CJ = \{J_1,\ldots,J_m\}$ corresponding to $\ovx$ and $E_{\leq t}$ (i.e. $\CJ$ is not necessarily a cluster structure with respect to $E$). The following variation on Lemma \ref{le:shift-blocks2} follows easily from the preceding discussion.

\begin{lemma}\label{le:shift-blocks3}
Let $E$ be a system of type \eqref{eq:semenov8} as described above (i.e. where $1 \leq t < s$), and let $\ovx$ be a solution of $E$. Let $\{J_1,\ldots,J_m\}$ be a cluster structure corresponding to $\ovx$ and $E_{\leq t}$, and let $N = \lcm(n_{t+1},\ldots,n_s)$. Then
$$
\ovx+N\beta_1\Delta_{J_1}+\dots+N\beta_m\Delta_{J_m}
$$
is a solution for $E$ with the same cluster structure
for any $\beta_1,\ldots,\beta_m\in\MZ$.
\end{lemma}

\begin{theorem}\label{th:non-homogeneous-system-with-1}
If a system $E$ of type \eqref{eq:semenov8} has a solution,
then it has a solution $\ovx$ (with the same cluster structure) satisfying
\begin{equation}
\label{eq:Semenov-bound-general3}
0 \leq x_1,\ldots,x_k < N + \mspn(E_{\leq t})
\end{equation}
where $N = \lcm(n_{t+1},\ldots,n_s)$.
\end{theorem}

\begin{proof}
Suppose that $\ovy$ is a solution to $E$, and set $\beta_i$ so that $0 \leq \min(\ovy_{J_i})+N\beta_i  < N$ (i.e. reduce $\min(\ovy_{J_i})$ modulo $N$). Then it follows from the definition of $\mspn(E_{\leq t})$ that $$\ovx = \ovy + N\beta_1\Delta_{J_1}+\dots+N\beta_m\Delta_{J_m}$$
satisfies \eqref{eq:Semenov-bound-general3}, and by Lemma \ref{le:shift-blocks3} it is also a solution to $E$.
\end{proof}

By \eqref{eq:N-bound}, Corollary \ref{co:mspan-poly} and Lemma \ref{le:check-sol}, a solution $x_1,\ldots,x_k$ for \eqref{eq:semenov8}
satisfying \eqref{eq:Semenov-bound-general3} is an $\NP$-certificate; thus, we have proved the following.

\begin{corollary}\label{co:Semenov-NP3}
The Diophantine problem for systems \eqref{eq:semenov8}
belongs to $\NP$.
\end{corollary}
Finally, from Corollary
\ref{co:semenov-NP-hard} we obtain
\begin{corollary}\label{co:Semenov-NP-complete2}
The Diophantine problem for systems \eqref{eq:semenov8}
is $\NP$-complete.
\end{corollary}

\section{Structure of the solution set}
\label{se:semenov-solution-set}

We say that a set $S\subseteq \MZ^k$ is \emph{semilinear} if
$S$ is a finite union of cosets, i.e.,
$$
S=\bigcup_{i=1}^n (\ovdelta_i+A_i),\ \
\mbox{ for some $\ovdelta_1,\ldots,\ovdelta_n \in\MZ^k$ and
$A_1,\ldots,A_n \le \MZ^k$.}
$$
Let $\ovx=(x_1,\ldots,x_k)$ be a solution for a system $E$ of type \eqref{eq:semenov8}. Following Section \ref{se:gen-sys}, let $\CJ=\{J_1,\ldots,J_m\}$ be a cluster structure corresponding to $\ovx$ and $E_{\leq t}$, and let us define
$$N_E = \begin{cases}
1 &\mbox{ if }t = s \\
\lcm(n_{t+1},\ldots,n_s) &\mbox{ otherwise}.
\end{cases}$$
The set $\CJ$ defines the tuples $\Delta_{J_1},\ldots,\Delta_{J_m}\in\{0,1\}^k$, so by Lemma \ref{le:shift-blocks3} the pair $(\ovx,\CJ)$ defines the following set of solutions for $E$
$$
S_{\ovx,\CJ}=
\Set{\ovx+N_E\beta_1\Delta_{J_1}+\dots+N_E\beta_m\Delta_{J_m}}{\beta_1,\ldots,\beta_m\in\MZ}
\ \subseteq\ \MZ^k
$$
(note that $S_{\ovx,\CJ}$ is a coset of $\MZ^k$). Define $\CT(E)$ to be the set of all pairs $(\ovx,\CJ)$ such that
\begin{itemize}
\item $\ovx$ is a solution for $E$;
\item $\CJ$ is a cluster structure for $\ovx$ and $E_{\leq t}$;
\item (boundedness)
$\ovx=(x_1,\ldots,x_k)$ satisfies \eqref{eq:Semenov-bound-general3}.
\end{itemize}
By construction, $\CT(E)$ is finite.

\begin{proposition}[Completeness]\label{pr:completeness}
The set of all solutions of a system $E$ of type \eqref{eq:semenov8} is equal to
\begin{equation*}\label{eq:sol-set}
S(E) = \bigcup_{(\ovx,\CJ)\in\CT(E)} S_{\ovx,\CJ}.
\end{equation*}
\end{proposition}

\begin{proof}
By Lemma \ref{le:shift-blocks3}, $\bigcup_{(\ovx,\CJ)\in\CT} S_{\ovx,\CJ}$
is a set of solutions for $E$.
Conversely, if $\ovy$ is a solution for $E$ and $\CJ$ is a cluster structure for $\ovx$ and $E_{\leq t}$, then (as in the proof of Theorem \ref{th:non-homogeneous-system-with-1}) the set $S_{\ovy,\CJ}$ contains a solution $\ovx$ satisfying
\eqref{eq:Semenov-bound-general3} with the same cluster structure $\CJ$.
Hence, $(\ovx,\CJ)\in\CT(E)$ and $\ovy\in S_{\ovy,\CJ}=S_{\ovx,\CJ}$, proving that $\ovy\in S(E)$.
\end{proof}

\begin{corollary}
The set of all solutions of \eqref{eq:semenov8} is semilinear.
\end{corollary}


With minor adjustments, the foregoing arguments apply to non-homogeneous systems as well. Specifically, if $E$ is a non-homogeneous system, then we form the associated homogeneous system $E'$ with auxiliary variable $x_0$, as in \eqref{eq:semenov6}. We now consider the set of $\ovx$ which are solutions to both $E$ and $E'$, and cluster structures $\CJ$ asscociated to $\ovx$ and $E_{\leq t}'$, where each $\CJ$ is of the form $\{J_1,\ldots,J_k,J_0\}$ and $0\in J_0$. We further stipulate that only the clusters $J_1,\ldots,J_k$ can be shifted (but not $J_0$). Hence, the definition of $S_{\ovx,\CJ}$ remains the same, and the proof of Proposition \ref{pr:completeness} goes through unchanged.

\section{Parameterized complexity of the Diophantine problem}
\label{se:semenov-parametrized-complexity}

In this section, we show that the Diophantine problem for
systems \eqref{eq:semenov8} can be solved in polynomial time if the number
of variables $k$ is bounded by a fixed constant. In other words, the Diophantine problem for \eqref{eq:semenov8} is \emph{fixed-parameter tractable}.

Let $q$ be a polynomial such that the time complexity for validating
an $\NP$-certificate $\ovx = (x_1,\ldots,x_k)$ for a system $E$
of type \eqref{eq:semenov8} (i.e. for checking that $\ovx$ is a solution) is $O(q(\size(E)))$, and let $p$ be a polynomial such that
$$
\mspn(E) < p(\size(E)).
$$
By Corollaries \ref{co:Semenov-NP3} and \ref{co:mspan-poly},
such polynomials exist.

\begin{proposition}\label{pr:complexity}
There exists an algorithm that
decides if a given system $E$ of type \eqref{eq:semenov8}
has a solution in time
\begin{equation}\label{eq:k-poly-bound}
O\rb{2^{k^2}\!\size^{2k^2}(E)\cdot p(\size(E))^k\cdot q(\size(E))}.
\end{equation}
\end{proposition}

\begin{proof}
Using \eqref{eq:N-exp-bound}, we see that the total number of $\NP$-certificates
$\ovx = (x_1,\ldots,x_k)$ satisfying \eqref{eq:Semenov-bound-general3} is
$$(N_E+\mspn(E_{\leq t}))^k < (2^k\size^{2k}(E) + p(\size(E)))^k,$$
which is $O\rb{2^{k^2}\!\size^{2k^2}(E)\cdot p(\size(E))^k}$.
It is straightforward to enumerate all such certificates, and the time required to check each one is $O(q(\size(E)))$, so the result follows.
\end{proof}

\begin{corollary}\label{co:bounded-Semenov-polynomial-time}
Fix $K\in\MN$.
The Diophantine problem
for systems \eqref{eq:semenov8}
with the number of variables bounded by $K$
can be solved in polynomial time.
Hence, the Diophantine problem
for systems \eqref{eq:semenov8}
is a fixed-parameter tractable problem.
\end{corollary}

\begin{proof}
The complexity bound \eqref{eq:k-poly-bound} of Proposition \ref{pr:complexity} is polynomial if $k$ is bounded by
the given constant $K$.
\end{proof}


\end{document}